\documentclass[a4paper]{article}
\usepackage{amsthm,amssymb,amsmath,enumerate,graphicx}
\newtheorem{definition}{Definition}

\newtheorem{theorem}[definition]{Theorem}

\newtheorem{lemma}[definition]{Lemma}

\newcommand{\comment}[1]{}
\newcommand{\C}{\mathcal C}

\newcommand{\emtext}[1]{\text{\em #1}}

\newcommand{\sm}{\setminus}
\newcommand{\id}{\text{\rm id}}
\newcommand{\rk}{\text{\rm rk}}

\title{Twins of rayless graphs}
\author{Anthony Bonato \and Henning Bruhn \and Reinhard Diestel \and Philipp Spr\"ussel}
\date{}

\begin{document}
\maketitle

\begin{abstract}
  Two non-isomorphic graphs are twins if each is isomorphic to a subgraph of the
  other. We prove that a rayless graph has either infinitely many
  twins or none.
\end{abstract}

\section{Introduction}

Up to isomorphism, the subgraph relation $\subseteq$ is antisymmetric
on finite graphs: If a finite graph $G$ is (isomorphic to) a subgraph
of $H$, i.e.\ $G\subseteq H$, and if also $H\subseteq G$, then $G$ and
$H$ are isomorphic. For infinite graphs this need no longer be the
case, see Figure~\ref{fig:stars}. Two non-isomorphic
graphs $G$ and $H$ are \emph{weak twins} if $G$ is isomorphic to a
subgraph of $H$ and vice versa, and \emph{strong twins} if both these
subgraph embeddings are induced. When $G$ and $H$ are trees the two
notions coincide, and we just speak of \emph{twins}.

\begin{figure}[ht]
  \centering
  \includegraphics{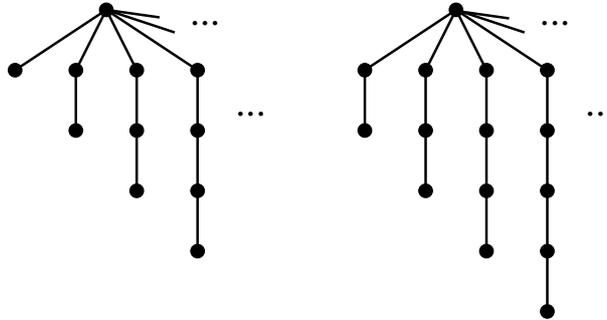}
  \caption{Each of the two graphs is a subgraph of the other.}
  \label{fig:stars}
\end{figure}

The trees in Figure~\ref{fig:stars} are twins, and by deleting some
of their leaves we can obtain infinitely many further trees that are
twinned with them. On the other hand, no tree is a twin of the infinite star.
Bonato and Tardif~\cite{BonTar06} conjectured that every tree is
subject to this dichotomy: that it has either infinitely many trees as twins or none. They call this the {\em tree alternative conjecture}.

In this paper we prove the corresponding assertion for \emph{rayless} graphs, graphs that contain no infinite path:

\begin{theorem}
  \label{mainthm}
  The following statements hold with both the weak and the strong
  notion of `twin'.
  \begin{enumerate}
  \item
  	A rayless graph has either infinitely many twins or none.
  \item
  	A connected rayless graph has either infinitely many connected
  	twins or none.
  \end{enumerate}
\end{theorem}

\noindent
We do not know of any counterexamples to the corresponding statements for arbitrary graphs, rayless or not.

\medbreak

Note that the `strong twin' version of Theorem~\ref{mainthm} does not
directly imply the `weak twin' version. Indeed, consider the complete
bipartite graph $K_{2,\infty}$ with one partition class consisting of
two and the other of (countable-)infinitely many vertices. By deleting any edge we
obtain a weak twin of $K_{2,\infty}$. However, it is straightforward
to check that $K_{2,\infty}$ has no strong twin.

We have stressed in the theorem that for a connected rayless graph we
may restrict ourselves to twins that are also connected. This is indeed
a stronger statement: For example, an infinite star has disconnected
weak twins---add isolated vertices---but no connected ones. We do not
know whether the same can occur for strong twins.

\medbreak
Twins were first studied in~\cite{BonTar03}.
The tree alternative conjecture was formulated in~\cite{BonTar06},
where it was proved in the special case of rayless trees. (Note that
Theorem~\ref{mainthm} reproves this case.) Most of the
work there was spent on showing that the conjecture holds for
\emph{rooted} rayless trees, which motivated Tyomkyn~\cite{Tyomkyn08}
to verify the conjecture for arbitrary rooted trees. Moreover, Tyomkyn
established the tree alternative conjecture for certain types of
locally finite trees. (A graph is \emph{locally finite} if all its
vertices have finite degree.) A~proof of the conjecture for arbitrary
unrooted locally finite trees has remained elusive.

In~\cite{Tyomkyn08}, a slightly different approach is outlined as
well. If a graph $G$ has a twin, then mapping $G$ to that twin and
back embeds it as a proper subgraph in itself. Tyomkyn conjectures
that, with the exception of the ray, every locally finite tree that
is a proper subgraph of itself has infinitely many twins.

In this paper we consider only embeddings as subgraphs or induced
subgraphs, leading to weak or strong twins. It seems natural, however,
to ask a similar question for other relations on graphs, such as the
minor relation or the immersion relation. Does a graph always have
either infinitely many `minor-twins' or none at all? Conceivably, the
question of when a graph is a proper minor of itself, as is claimed
for countable graphs by Seymour's {\em self-minor conjecture}, should play a
role in this context. 
The self-minor conjecture is described in Chapter 12.5 in~\cite{DiestelBook05};
partial results are due to Oporowski~\cite{Opo90} and Pott~\cite{JulianSMC}.
In related work, Oporowski~\cite{Opo99}
characterises the minor-twins of the infinite grid, and
Matthiesen~\cite{Lillian06} studies a complementary question with
respect to the topological minor relation, restricted to rooted
locally finite trees.

\medskip
In the next section we introduce a recursive technique for handling
rayless graphs, which we will use in Section~\ref{proofs} to prove
Theorem~\ref{mainthm}.

\section{A rank function for rayless graphs}

All our graphs are simple. For general graph theoretical concepts
and notation we refer the reader to~\cite{DiestelBook05}.

Our proof of Theorem~\ref{mainthm} is based on a construction by
Schmidt~\cite{schmidt83} (see also Halin~\cite{halin98} for an
exposition in English) that assigns an ordinal $\rk(G)$, the
\emph{rank of $G$}, to all rayless graphs $G$ as follows:

\begin{definition}\label{rank}
  Let $\rk(G)=0$ if and only if $G$ is a finite graph. Then
  recursivley for ordinals $\alpha>0$, let $\rk(G)=\alpha$ if and only
  if
  \begin{enumerate}
  \item
    $G$ has not been assigned a rank smaller than $\alpha$; and
  \item\label{separate}
    there is a finite set $S\subseteq V(G)$ such that every component
    of $G-S$ has rank smaller than $\alpha$.
  \end{enumerate}
\end{definition}

It is easy to see that the graphs that receive a rank are precisely
the rayless ones.
The rank function makes the class of rayless graphs accessible to induction
proofs. One of the first applications of the rank was the proof of the
reconstruction conjecture restricted to rayless trees by Andreae and
Schmidt~\cite{AnSch84}. Recently, the rank was used to verify the
unfriendly partition conjecture for rayless graphs,
see~\cite{unfrayless}.

We shall need a few properties of the rank function that are either simple
consequences of the definition or can be found in~\cite{halin98}.
Let $G$ be an infinite rayless graph, and let $S$ be minimal among the
sets as in~\ref{separate} of Definition~\ref{rank}. It is not hard to
see that $S$ is unique with this property. We call $S$ the
\emph{kernel of $G$} and denote it by $K(G)$. Furthermore, it holds
that:
\begin{itemize}
\item
  if $H$ is a subgraph of $G$, then $\rk(H) \le \rk(G)$; and
\item
  if $G$ is connected, then $K(G)$ is non-empty; and
\item
  $\rk(G-X)=\rk(G)$ for any finite $X\subseteq V(G)$.
\end{itemize}
In particular, if $C$ is a component of $G-K(G)$, then $G[C\cup K(G)]$
has smaller rank than $G$.

To illustrate the definition of the rank, let us note that an infinite
star has rank~$1$, and its kernel consists of its centre. The same
holds for the graphs in Figure~\ref{fig:stars}. On the other hand, the
disjoint union of infinitely many infinite stars (or in fact, of any
graphs of rank $1$) has rank~$2$ and an
empty kernel.

\section{The proofs}
\label{proofs}

In this section we prove the `strong twin' version of Theorem~\ref{mainthm}.
All proofs will apply almost literally to the case of weak twins
instead of strong twins. For that reason we will often drop the
qualifiers `strong' and `weak'.

Let $G,H$ be two rayless graphs and let $X\subseteq V(G)$ and
$Y\subseteq V(H)$ be finite vertex subsets. 
We call a homomorphism
$\phi:G\to H$ a \emph{strong embedding of $(G,X)$ in $(H,Y)$} 
if it is injective, $\phi(G)$ is an induced subgraph of $H$, and
$\phi(X)\subseteq Y$. 
Alternatively, we shall say that $\phi:(G,X)\to (H,Y)$ is 
a strong embedding.
Observe that  $\phi$  preserves edges as well as non-edges. 
We call $(G,X)$ and
$(H,Y)$ \emph{isomorphic} if there is an isomorphism $\gamma:(G,X)\to
(H,Y)$, i.e.\ if $\gamma$ is a graph-isomorphism between $G$ and $H$ with $\gamma(X)=Y$. 
We say that $(G,X)$ and $(H,Y)$ are
\emph{strong twins} if they are not isomorphic and there exist strong
embeddings $\phi:(G,X)\to (H,Y)$ and $\psi:(H,Y)\to (G,X)$; note that $\phi(X)= Y$ and
$\psi(Y)=X$ in this case. For $(G,X)$ and $(H,Y)$ to be \emph{weak
twins} we only require $\phi$ and $\psi$ to be injective homomorphisms
with $\phi(X)= Y$ and $\psi(Y)=X$. Let us point out that rayless
graphs $G$ and $H$ are (strong resp.\ weak) twins if and only if the
tuples $(G,\emptyset)$ and $(H,\emptyset)$ are (strong resp.\ weak)
twins.

\medskip

As we have noted, subgraphs of rayless graphs do not have larger rank.
Moreover, if a subgraph $G'$ of a rayless graph $G$ has the same rank
as $G$, then $K(G')\subseteq K(G)$ since $K(G)\cap V(G')$ is a set as
in~\ref{separate} of Definition~\ref{rank}. We thus have:

\begin{lemma}\label{fixedkernels}
  Let $G$ and $H$ be rayless graphs, and let there be injective
  homomorphisms $\phi:G\to H$ and $\psi:H\to G$. Then $\phi(K(G))=
  K(H)$ and $\psi(K(H))=K(G)$.
\end{lemma}
In particular, the lemma implies that if $(G,X)$ and $(H,Y)$ are 
twins, then $(G,X\cup K(G))$ and $(H,Y\cup K(H))$ are twins too.
\medbreak

Let $G$ and $H$ be rayless graphs, and let $X\subseteq V(G)$ and
$Y\subseteq H$ be finite vertex sets. We write $\bar X$ as a
shorthand for $X \cup K(G)$, and define $\bar Y$ analogously. Assume
there are (strong) embeddings $\phi:(G,X)\to (H,Y)$ and $\psi:(H,Y)\to (G,X)$ and set
$\iota:= \psi\circ\phi$. Since, by Lemma~\ref{fixedkernels}, $\iota$
induces an automorphism on (the subgraph induced by)
the finite set $\bar X$ there exists a $k$ with
$\iota^k\restriction\bar X=\id_{\bar X}$. By replacing $\phi$ with
$\phi\circ\iota^{k-1}$, we may assume that
\begin{equation}\label{idonX}
  \begin{minipage}[c]{0.8\textwidth}\em
    $\phi:(G,X)\to (H,Y)$ and $\psi:(H,Y)\to (G,X)$ are embeddings so
    that the restriction of $\iota=\psi\circ\phi$ to $\bar X$
    coincides with $\id_{\bar X}$. 
  \end{minipage}\ignorespacesafterend
\end{equation}

Assume now that $(G,X)$ and $(H,Y)$ are isomorphic by virtue of an
isomorphism $\gamma$. In that case, abusing symmetry and notation,
let us write $(G,X)\simeq_\eta (H,Y)$, where $\eta$ denotes the
isomorphism $X\to Y$ induced by $\gamma$. Denote by $\C_G$ the set of
all subgraphs $G[C\cup\bar X]$ of $G$, where $C$ is a component of
$G-\bar X$. For $A\in\C_G$ set
\begin{equation*}
  I_G(A):=\{D\in\C_G : (D,\bar X)\simeq_\id(A,\bar X)\}.
\end{equation*}

\begin{lemma}
  \label{isolem}
  Let $G$ and $H$ be rayless graphs, and let $X\subseteq V(G)$ and
  $Y\subseteq V(H)$ be finite. The following statements are
  equivalent.
  \begin{enumerate}
  \item\label{isotuples}
	$(G,X)$ and $(H,Y)$ are isomorphic.
  \item\label{compbij}
  	There is a bijection $\alpha:\C_G\to\C_H$ and an isomorphism
  	$\eta:G[\bar X]\to G[\bar Y]$ with $\eta(X)=Y$ so that
  	$(A,\bar X)\simeq_{\eta}(\alpha(A),\bar Y)$ and $|I_G(A)|=
  	|I_H(\alpha(A))|$ for all $A\in\C_G$.
  \end{enumerate}
  Moreover, if~\ref{isotuples} and~\ref{compbij} hold, then $\alpha$,
  $\eta$, and the isomorphism $\phi:(G,X)\to (H,Y)$ can be
  chosen so that $\phi\restriction\bar X=\eta$ and $\phi(A)=\alpha(A)$
  for every $A\in\C_G$.
\end{lemma}

\begin{proof}
  First assume that \ref{isotuples} holds and let $\phi:(G,X)\to(H,Y)$
  the isomorphism certifying this fact. Put $\eta:=\phi\restriction
  \bar X$. Observe that, by Lemma~\ref{fixedkernels}, for
  every $A\in\mathcal C_G$ there is a $B\in\C_H$ with $\phi(A)=B$; set
  $\alpha(A):=B$. Clearly, $\alpha$ is a bijection and $(A,\bar X)
  \simeq_\eta(\alpha(A),\bar Y)$. It remains to show that
  $|I_G(A)|=|I_H(\alpha(A))|$ for all $A\in\C_G$. Indeed, for every $C\in
  I_G(A)$ we have $\alpha(C)\in I_H(\alpha(A))$: Since $(A,\bar X)\simeq_\id
  (C,\bar X)$, by virtue of an isomorphism $\gamma$ say, $\phi\circ
  \gamma\circ\phi^{-1}$ is an isomorphism certifying $(\alpha(A),\bar Y)
  \simeq_\id(\alpha(C),\bar Y)$. Hence we obtain $|I_H(\alpha(A))| \ge |I_G(A)|$ and
  analogously $|I_G(A)| \ge |I_H(\alpha(A))|$.
  
  Now assume that \ref{compbij} holds. Then for every $A\in\C_G$ there
  is an isomorphism $\phi_A:A\to\alpha(A)$ that witnesses $(A,\bar X)
  \simeq_\eta  (\alpha(A),\bar Y)$. Now the function $\phi:G\to H$
  defined by $\phi\restriction A:= \phi_A$ for every $A$ is an
  isomorphism of $(G,X)$ and $(H,Y)$ satisfying $\phi\restriction\bar X=\eta$ and $\phi(A)=\alpha(A)$ for every $A\in\C_G$.
\end{proof}

We call the tuple $(G,X)$ \emph{connected} if $G-X$ is connected.

\begin{lemma}\label{mainlem}
  Let $(G,X)$ and $(H,Y)$ be strong twins, where $G$ and $H$ are
  rayless graphs, and $X\subseteq V(G)$ and $Y\subseteq V(H)$ finite.
  Then $(G,X)$ has infinitely many strong twins. If both $(G,X)$ and
  $(H,Y)$ are connected, then $(G,X)$ has infinitely many connected
  strong twins.
\end{lemma}

Before we prove the lemma let us remark that it immediately implies 
the strong version of Theorem~\ref{mainthm} if we set $X=Y=\emptyset$.

\begin{proof}[Proof of Lemma~\ref{mainlem}]
  We proceed by transfinite induction on the rank of $G$. For rank $0$
  the statement is trivially true as finite graphs do not have twins.
  We may thus assume that $G$ has rank $\kappa>0$ and that the lemma is true
  for rank smaller than~$\kappa$.
  
  Assume there exists a $C_0\in\C_G$ so that $(C_0,\bar X)$ has a
  connected twin. Then, as $C_0$ has rank smaller than $\kappa$, the
  inductive hypothesis provides us with infinitely many connected
  twins $(C_i,X_i),~i>0$, of $(C_0,\bar X)$. By applying~\eqref{idonX}
  to $(C_0,\bar X)$ and $(C_i,X_i)$ we may assume that the restrictions
  to $\bar X$ and $X_i$, respectively, of the mutual embeddings are
  inverse isomorphisms. Hence, by identifying $X_i$ with $\bar X$ by
  this isomorphism we may assume that the twins have the
  form $(C_i,\bar X)$ and that the corresponding embeddings
induce the identity on $\bar X$.
  Denote by $\mathcal T$
  the set of $C\in\mathcal C_G$ for which either $(C,\bar X)\simeq_\id
  (C_0,\bar X)$, or for which $(C,\bar X)$ is a twin of $(C_0,\bar X)$
  by virtue of mutual embeddings that each induce the identity on $\bar X$.
  For every~$i\in\mathbb N$, define $G_i$ to be the graph obtained
  from $G$ by replacing every $C\in\mathcal T$ by a copy of $C_i$. 

The construction ensures two properties. First, 
there are strong embeddings $(G,X)\to (G_i,X)$ and $(G_i,X)\to (G,X)$
for every $i$. So, if infinitely many of the $(G_i,X)$ are non-isomorphic,
we have found infinitely many twins of $(G,X)$.
Second, 
for $j\neq k$
it follows that $|I_{G_k}(C_j)|=0\not= |I_{G_j}(C_j)|$. Consequently,
Lemma~\ref{isolem} implies
\begin{equation}
\label{notid}
(G_j,\bar X)\not\simeq_\id (G_k,\bar X).
\end{equation} 

Assume that for distinct $i,j,k$ the tuples
$(G_i,X)$, $(G_j,X)$ and $(G_k,X)$ are isomorphic. Thus, by Lemma~\ref{isolem} there are
isomorphisms $\eta$
between $\bar X\subseteq V(G_i)$ and $\bar X\subseteq V(G_j)$
and $\eta'$ between $\bar X\subseteq V(G_i)$ and $\bar X\subseteq V(G_k)$
so that $(G_i,\bar X)\simeq_\eta (G_j,\bar X)$ and 
$(G_i,\bar X)\simeq_{\eta'} (G_k,\bar X)$.
Now, if $\eta=\eta'$, then the resulting isomorphism between $G_j$ and $G_k$
would induce the identity on $\bar X$, which is impossible by~\eqref{notid}.
  As there are only finitely many automorphisms of the
  finite set $\bar X$, we deduce that each $(G_i,X)$ is isomorphic to only finitely many
  $(G_j,X)$. Therefore we can easily find among the $(G_i,X)$
  infinitely many that are pairwise non-isomorphic.
  
  Finally, we claim that if $(G,X)$, i.e.\ $G-X$, is connected, then so is each
  $(G_i,X)$, i.e.\ $G_i-X$. Indeed, by construction there is an embedding
  $(G,X) \to (G_i,X)$ that restricts to the identity on $\bar X$ and whose image meets
  all components of $G_i-\bar X$. As $G-X$ is connected, as well as each
  component of $G_i-\bar X$, we deduce that $G_i-X$ is connected.

  Thus, we may assume from now on that
  \begin{equation}\label{isotwin}
  	\emtext{
  	  for each $C\in\C_G$, $(C,\bar X)$ has no connected twin.
	}
  \end{equation}
  By symmetry, the same holds for $(H,Y)$.
  
Let $\phi:(G,X)\to (H,Y)$ and $\psi:(H,Y)\to (G,X)$ be strong embeddings, and 
recall that by~\eqref{idonX} we may assume that 
$\iota:=\psi\circ\phi$ induces the identity map on $\bar X$.
  By Lemma~\ref{isolem} and symmetry, we may assume that for $\eta:=
  \phi\restriction\bar X$ there are $A\in\C_G$ and $B\in\C_H$ with
  $(A,\bar X)\simeq_\eta(B,\bar Y)$ so that $|I_G(A)|>|I_H(B)|$.
  
  Observe that by Lemma~\ref{fixedkernels}
  \begin{equation}\label{components}
	\emtext{
	  for every $C\in\C_G$ there is a (unique) $D\in\C_G$ with
	  $\iota(C)\subseteq D$.
	}
  \end{equation}
  Furthermore, we point out that $\iota$ is a strong self-embedding of
  $(G,X)$, and also of $(G,\bar X)$.
  
  We define a directed graph $\Gamma$ on $\C_G$ as vertex set by
  declaring $(C,D)$ to be an edge if $\iota(C)\subseteq D$ for $C,D
  \in\C_G$. We do allow $\Gamma$ to have loops and parallel edges
  (which then, necessarily, are pointing in opposite directions).
  Note that by~\eqref{components} every vertex in $\Gamma$ has out-degree one. Define
  $\mathcal A$ to be the set of those $A'\in I_G(A)$ for which the
  unique out-neighbour does not lie in $I_G(A)$.
  
  Suppose that distinct $A_1,A_2\in I_G(A)$ are mapped by $\phi$ into
  the same $B'\in\C_H$. If $A_1$ (and then also $A_2$) is finite, then
  $|V(B')|>|V(A_i)|$ for $i=1,2$ since the injectivity of $\phi$
  implies $\phi(A_1)\cap \phi(A_2) = \bar Y$. Consequently, we obtain
  $B'\notin I_H(B)$. Let now $A_1$ and  $A_2$ be infinite. Unless
  $\rk(B')>\rk(A_1)=\rk(A_2)$ it follows that $\phi(K(A_i))\subseteq
  K(B')$ for $i=1,2$. Since $A_1-\bar X$ and $A_2-\bar X$ are
  connected the kernels $K(A_i-\bar X)$ are non-empty (but finite).
  Again from $\phi(A_1)\cap \phi(A_2) = \bar Y$ we obtain that $K(B')$
  has larger cardinality than either of $K(A_1)$ and $K(A_2)$, which
  implies $B'\notin I_H(B)$. Therefore, we have in all cases that
  $B'\notin I_H(B)$. Since~\eqref{isotwin} and~\eqref{components}
  necessitate that $\phi(A')$
  is contained in an element of $I_H(B)$ for every $A'\in I_G(A) \sm
  \mathcal A$ we deduce that $|\mathcal A|\ge |I_G(A)|-|I_H(B)|$.
  Thus, it holds that
  \begin{equation}\label{freespace}
	\emtext{
	  $\mathcal A\neq\emptyset$, and if $I_G(A)$ is infinite, then we have
	  $|\mathcal A|=|I_G(A)|$.
	}
  \end{equation}
  
  \noindent
  By construction, the set $\mathcal A$ is independent in $\Gamma$.
  Moreover,
  \begin{equation}\label{indacyc}
	\begin{minipage}[c]{0.8\textwidth}\em
	  there is no directed path in $\Gamma$ starting in $\mathcal A$
	  and ending in $I_G(A)$, and there is no directed cycle containing any
	  $A'\in\mathcal A$.
	\end{minipage}\ignorespacesafterend 
  \end{equation} 
  To prove~\eqref{indacyc}, suppose that $C_1,\ldots,C_k$ is a directed
  path in $\Gamma$ with $C_1\in\mathcal A$ and $C_k\in I_G(A)$ (possibly
  even $C_k\in\mathcal A$). Since repeated application of $\iota$
  maps every $(C_1,\bar X)$ into any $(C_i,\bar X)$ and likewise
  $(C_i,\bar X)$ into $(C_k,\bar X)\simeq_\id (C_1,\bar X)$, we deduce
  that $(C_i,\bar X)\simeq_\id (C_j,\bar X)$ for $i,j\in\{1,\ldots,k\}$, as they cannot be
  twins by~\eqref{isotwin} (recall that $\iota\restriction\bar X=\id_{\bar X}$ by~\eqref{idonX}). 
However, $(C_1,\bar X)\simeq_\id (C_2,\bar X)$ violates $C_1\in\mathcal A$.
The same arguments hold if $C_1,\ldots,C_k$ is a
  directed cycle that meets~$\mathcal A$.
  
  Define $\mathcal A^-$ to be the set of all $C\in\mathcal C_G$ from
  which there is a nontrivial directed path in $\Gamma$ ending in
  $I_G(A)$ (in particular, $I_G(A)\sm\mathcal A \subseteq \mathcal
  A^-$). Setting $\mathcal A^+:=\mathcal C_G\setminus(\mathcal
  A\cup\mathcal A^-)$ we see with~\eqref{indacyc} that $(\mathcal A^-,
  \mathcal A,\mathcal A^+)$ partitions~$\mathcal C_G$. By definition,
  the out-neighbour of an $A'\in\mathcal A$ does not lie in $\mathcal
  A$, and by~\eqref{indacyc} the out-neighbour does not lie in
  $\mathcal A^-$ either. Hence, we have $\iota(A')\subseteq\bigcup
  \mathcal A^+$. On the other hand, the definition of $\mathcal A^-$
  implies that the out-neighbour of every $A^+\in\mathcal A^+$ is
 contained in $\mathcal A^+$. Thus it follows that $\iota(A^+)
  \subseteq \bigcup\mathcal A^+$. In summary, we obtain:
  \begin{equation}\label{plusminus}
	\emtext{
	  $(\mathcal A^-,\mathcal A,\mathcal A^+)$ partitions~$\mathcal
	  C_G$ and $\iota\left(\bigcup\mathcal A\cup\bigcup\mathcal A^+\right)
	  \subseteq\bigcup\mathcal A^+.$
	}
  \end{equation}
  
  We claim that there exists a strong self-embedding $\gamma:(G,X)\to
  (G,X)$ that induces the identity on $\bigcup I_G(A)\sm\mathcal A$ (in
  particular on $\bar X$) and satisfies
  \begin{equation}\label{makefree}
	\emtext{
	  $\gamma(G)\cap\bigcup\mathcal A=\bar X$.
	}
  \end{equation}
  On $\bar X$ we define $\gamma$ to be the identity. For every other
  vertex $v\in V(G)$ we consider the unique $C\in\C_G$ containing $v$.
  If $C\in\mathcal A^-$ we set $\gamma(v):=v$, and if $C\in\mathcal A
  \cup\mathcal A^+$ we put $\gamma(v):=\iota(v)$. Note that
  by~\eqref{components} it holds that for every $C\in\C_G$ we have $\gamma
  \restriction C=\id_C$ or $\gamma\restriction C=\iota\restriction C$.
  It is immediate from~\eqref{plusminus} that~\eqref{makefree} holds.
  Moreover, since the identity as well as $\iota$ are strong
  self-embeddings it follows from~\eqref{plusminus} that $\gamma$ is
  one, too.
  
  If $I_G(A)$ is infinite, then by~\eqref{freespace} we change $\gamma$
  on each component in $I_G(A)\sm\mathcal A$ so as to obtain a strong
  self-embedding $\varphi$ whose image avoids $\bigcup I_G(A)\sm\mathcal
  A - \bar X$. Then $\beta:=\varphi^2$ is a strong self-embedding that
  induces the identity on $\bar X$ and satisfies
  \begin{equation}\label{makefreeinf}
	\emtext{
	  $\beta(G)\cap\bigcup I_G(A)=\bar X$.
	}
  \end{equation}
  
  \medbreak
  
  Let us now construct infinitely many strong twins of $(G,X)$. Assume
  first that $I_G(A)$ is a finite set. Add a disjoint copy $\tilde A$
  of $A$ to $G$ and identify every vertex in $\bar X$ with its copy in
  $\tilde A$. The resulting graph $G_1$ is clearly a supergraph of
  $G$. But by~\eqref{makefree} we can also embed $(G_1,X)$ in $(G,X)$:
  extend $\gamma$ to an
  embedding of $(G_1,X)$ in $(G,X)$ by mapping $\tilde A-\bar X$ to $A'-
  \bar X$ for some $A'\in\mathcal A$. Here, we use that $\mathcal A
  \neq\emptyset$, by~\eqref{freespace}. Note that $|I_{G_1}(A)|=
  |I_G(A)|+1$. Now we repeat this process, with $G_1$ in the role of
  $G$, so as to obtain $G_2$, and so on. Since $|I_{G_i}(A)| \not=
  |I_{G_j}(A)|$ for all $i \not= j$, we can deduce from
  Lemma~\ref{isolem}, as in the proof of \eqref{isotwin}, that
  each $(G_i,X)$ is isomorphic to only finitely
  many $(G_j,X)$. Therefore we can find among the $(G_i,X)$ infinitely
  many twins of $(G,X)$.
  
  So, consider the case when $I_G(A)$ contains infinitely many
  elements $A_1,A_2,\ldots$. Set $G_i:=G-(\bigcup I_G(A)\sm\{A_1,
  \ldots,A_i\}-\bar X)$ for $i\in\mathbb N$. Since,
  by~\eqref{makefreeinf}, $\beta$ can be used to embed $(G,X)$ in
  $(G_i,X)$ we can again find infinitely many twins of $(G,X)$---note
  that $|I_{G_i}(A)|$ takes different (finite) values.
  
  Finally, observe that in both cases, all the strong twins we
  constructed are connected if $(G,X)$ is.
\end{proof}

\bibliographystyle{amsplain}
\bibliography{collective}

\small
\vskip2mm plus 1fill
Version 19 Nov 2009
\bigbreak

\begin{tabular}{cc}
\begin{minipage}[t]{0.5\linewidth}
Anthony Bonato
{\tt <abonato@ryerson.ca>}\\
Ryerson University\\
350 Victoria Street\\
 Toronto, Ontario\\
 M5B 2K3\\
Canada
\end{minipage}
&
\begin{minipage}[t]{0.5\linewidth}
Henning Bruhn
{\tt <hbruhn@gmx.net>}\\
Reinhard Diestel\\
Phillip Spr\"ussel\\
{\tt <spruessel@math.uni-hamburg.de>}\\
Mathematisches Seminar\\
Universit\"at Hamburg\\
Bundesstra\ss e 55\\
20146 Hamburg\\
Germany
\end{minipage}
\end{tabular}

\end{document}